\documentclass[11pt]{article}
\usepackage{graphicx} % Required for inserting images
\usepackage{amsthm}
\usepackage{amssymb}
\usepackage[lite]{amsrefs}
\usepackage{amsmath}
\usepackage{mathtools}
\usepackage{amsfonts}
\usepackage{tikz-cd}
\usepackage{textcomp}
\usepackage{enumerate}
\newtheorem{theorem}{Theorem}
\newtheorem{proposition}{Proposition}
\newtheorem{lemma}{Lemma}
\newtheorem{remark}{Remark}
\newtheorem{corollary}{Corollary}
\newtheorem{definition}{Definition}

\newcommand{\hilb}{\mathcal{H}}

\newcommand{\B}{\mathcal{B}}
\newcommand{\C}{\mathbb{C}}
\newcommand{\R}{\mathbb{R}}
\newcommand{\mob}{{M\"{o}bius}}
\newcommand{\A}{\mathcal{A}}
\newcommand{\Z}{\mathbb{Z}}
\newcommand{\id}{\operatorname{id}}

\title{ {\huge The $\alpha$-induction of Graded Local Conformal Nets}}
\author{{\sc ZIYUN XU}\\{\small Graduate School of Mathematical Sciences}\\{\small The University of Tokyo, Komaba, Tokyo, 153-8914, Japan}\\{\small e-mail: {\tt zyxu0805@g.ecc.u-tokyo.ac.jp}}}
\begin{document}
\maketitle
\begin{abstract}
    The $\alpha$-induction of graded local conformal nets is studied. 
    We show that inclusions of graded local conformal nets give rise to braided subfactors so that the $\alpha$-induction is still effective for graded local conformal nets. As an application, we give a shorter proof of classification of $N=2$ superconformal nets in the discrete series.
\end{abstract}
\section{Introduction}

The subfactor theory initiated by Jones \cite{jones} has opened many surprising connections of operator algebras to various branches of mathematics and physics, such as
$3$-dimensional topology, quantum groups and statistical mechanics. Subfactor theory is a central aspect of the
representation theory of conformal nets \cite{fg}, \cite{wass}. Methods of subfactor theory have turned out to be powerful in the operator algebraic approach to conformal field theory. For instance, $\alpha$-induction was first introduced in \cite{lr} and studied in \cite{xu}, with a more general setting studied in \cite{be1}, \cite{be2}, \cite{be3}. In \cite{bek}, $\alpha$-induction was defined and studied in a fully general setting. By applying $\alpha$-induction to certain subfactors \cite{kl}, the authors gave a classification of conformal nets and two dimensional full conformal nets \cite{kl2} in the discrete series, respectively. \\
On the other hand, the AQFT (Algebraic Quantum Field Theory) method is still useful for superconformal field theories. The $N=1$ superconformal nets were first introduced and studied in \cite{ckl} and $N=2$ superconformal nets were introduced and studied in \cite{chklx}. In \cite{chklx}, the authors studied the $\alpha$-induction of the Bose subnet of a given $N=2$ superconformal net. By analyzing fermionization of the inclusion of Bose subnets, the authors gave a classification of $N=2$ superconformal nets in the discrete series.\\
In this paper, we first study the graded localized endomorphisms of a given graded local conformal net $\B$. As a main result, we show that the category of graded localized endomorphisms of $\B$ is a braided tensor category which allows the $\alpha$-induction to proceed. In addition, the simple extension procedure can be naturally generalized to graded local conformal nets. Finally, by directly applying the $\alpha$-induction to a given $N=2$ superconformal net in the discrete series, we give a shorter classification of $N=2$ superconformal nets in the discrete series according to Gannon's list \cite{gannonlist}. Our classification follows the strategy of \cite{chklx}. Given an $N=2$ superconformal net $\A$ in the discrete series, we have a net of subfactors $\A_c(I)\subset\A(I)$, where $\A_c$ is the $N=2$ super-Virasoro net. From one of those subfactors, we obtain a modular invariant by applying the $\alpha$-induction. On the other hand, given a modular invariant, we can construct the dual canonical endomorphism as $\theta=\bigoplus_{\lambda} Z_{0,\lambda}\lambda$. Then we construct a subfactor whose dual canonical endomorphism is $\theta$. Finally, we show that the above relation between $N=2$ superconformal nets in the discrete series and modular invariants is a one-to-one correspondence.

\section{Preliminaries}
\subsection{Braided subfactors and $\alpha$-induction}
  Let $N, M$ be infinite factors. We denote by Mor$(N,M)$ the set of unital $*$-homomorphisms from $N$ to $M$. The dimension of $\lambda\in$Mor$(N,M)$ is defined as $d_{\lambda} \coloneqq [M:\lambda(N)]^{\frac{1}{2}}$, where $[M:\lambda(N)]$ is the minimal index \cite{jones}, \cite{kosaki}. A morphism
$\lambda\in$Mor$(N,M)$ is called irreducible if $\lambda(N)\subset M$ is irreducible, i.e. $\lambda(N)'\cap M = \C$. Two morphisms $\lambda, \lambda'\in$Mor$(N,M)$ are called unitary equivalent if there is a unitary $u\in M$ such that $\lambda' = u^*\lambda u$. The unitary equivalence class $[\lambda]$ of a morphism $\lambda\in$Mor$(N,M)$ is called an $M$-$N$ sector. For sectors, we have a notion of direct sums, products and conjugates (see \cite{bek} and the references therein for more details). For $\lambda,\mu\in $Mor$(N,M)$ we denote 
Hom$(\lambda,\mu)=\{m\in M: m\lambda(n)=\mu(n)m, n\in N\}$
and$\langle\lambda,\mu\rangle=\operatorname{dim}\operatorname{Hom}(\lambda,\mu)$. Let $N$ be a type III factor equipped with a system $\Delta\subset$Mor$(N,N)$ of endomorphisms in the sense of (\cite{bek}, definition 2.1). This means morphisms in $\Delta$ are irreducible and have finite statistical dimensions and they are different as sectors. In addition, they form a closed fusion algebra. Then $\Sigma(\Delta)\subset$Mor$(N,N)$ denotes the set of morphisms which decompose as sectors into finite sums of elements in $\Delta$. We assume that $\Delta$ is braided in the sense of \cite{bek} and we extend the braiding to $\Sigma(\Delta)$. We then consider a braided subfactor $N\subset M$, i.e. they are both type III factors and that the dual canonical endomorphism sector $[\theta]$ decomposes in a finite sum of sectors of morphisms in $\Delta$, i.e. $\theta\in\Sigma(\Delta)$. Here $\theta = \bar{\iota}\iota$ with
$\iota: N \to M$ being the injection homomorphism and $\Bar{\iota}\in$Mor$(M,N)$ being a conjuagte morphism. Note that this forces the statistical dimension of $\theta$ and thus the index of the subfactor $N\subset M$ to be finite. Then we can define $\alpha$-induction by
    \[
    \alpha^{\pm}_{\lambda} = \Bar{\iota}^{-1}\circ\operatorname{Ad}(\epsilon(\lambda,\theta)^{\pm})\circ\lambda\circ\Bar{\iota}
    \]
for $\lambda\in\Sigma(\Delta)$. Then $\alpha^{\pm}_{\lambda} \in $Mor$(M,M)$ and satisfy $\alpha^{\pm}_{\lambda}|_N = \lambda$.

The ambichiral system ${}^{}_M \mathcal{X}^0_M\subset$Mor$(M,M) $ is the subset corresponding to subsectors of $[\alpha^+_{\lambda}]$ and $[\alpha^-_{\mu}]$ when $\lambda$ and $\mu$ vary in $\Delta$. In \cite{be3} section $3$, a relative braiding between representative endomorphisms was introduced. Namely, if $\omega,\nu\in$Mor$(M,M)$ are such that $[\omega]$ and $[\nu]$ are subsectors of $[\alpha_{\sigma}^+]$ and $[\alpha_{\rho}^-]$ respectively, for some $\sigma,\rho \in \Sigma(\Delta)$, then 

\[
\mathcal{E}_r^{\pm}(\omega,\nu) = s^*\alpha_{\rho}^-(t^*)\epsilon^{\pm}(\sigma,\rho)\alpha_{\sigma}^+(s)t \in \text{Hom}(\omega\nu,\nu\omega)
\]
is unitary where $t\in\text{Hom}(\omega,\alpha_{\sigma}^+), s\in\text{Hom}(\nu,\alpha_{\rho}^-)$ are isometries. It was shown that $\mathcal{E}_r^{\pm}(\omega,\nu)$ does not depend on $\sigma,\rho\in \Sigma(\Delta)$ and not on the isometries $s,t$,
in the sense that if $\lambda, \gamma \in \Sigma(\Delta)$ such that 
$[\alpha_{\lambda}^+]=[\alpha_{\sigma}^+]$, $[\alpha_{\gamma}^-]=[\alpha_{\rho}^-]$, and $a\in\text{Hom}(\omega,\alpha_{\lambda}^+), b\in\text{Hom}(\nu,\alpha_{\gamma}^-)$ are isometries, then 
\[
\mathcal{E}_r^{\pm}(\omega,\nu) = b^*\alpha_{\gamma}^-(a^*)\epsilon^{\pm}(\lambda,\gamma)\alpha_{\lambda}^+(b)a.
\]
It was shown in \cite{be3} proposition $3.12$ that the family of unitary relative braiding operators 
satisfies the following properties:
\begin{enumerate}[(A)]
    \item $\mathcal{E}_r^{\pm}(\text{id},\omega) = \mathcal{E}_r^{\pm}(\nu,\text{id})= 1$, for $\omega, \nu \in {}^{}_M \mathcal{X}^0_M$.
    \item (Composition rules): 
    \begin{align*}
\mathcal{E}_r^{\pm}(\omega\nu,\kappa) &= \mathcal{E}_r^{\pm}(\omega,\kappa)\omega(\mathcal{E}_r^{\pm}(\nu,\kappa)),\\
\mathcal{E}_r^{\pm}(\omega,\nu\kappa) &= \nu(\mathcal{E}_r^{\pm}(\omega,\kappa))\mathcal{E}_r^{\pm}(\omega,\nu),
    \end{align*}
    where $\omega,\nu,\kappa\in{}^{}_M \mathcal{X}^0_M $.
    
    \item (Naturality): Let $\omega,\omega',\nu,\nu'\in{}^{}_M \mathcal{X}^0_M$. If $a\in\text{Hom}(\omega,\omega')$ and $b\in\text{Hom}(\nu,\nu')$, then:
    \begin{align*}
        \nu(a)\mathcal{E}_r^{\pm}(\omega,\nu) &= \mathcal{E}_r^{\pm}(\omega',\nu)a, \\
        b\mathcal{E}_r^{\pm}(\omega,\nu) &= \mathcal{E}_r^{\pm}(\omega,\nu')\omega(b).
    \end{align*}
\end{enumerate}
Let $\kappa,\omega,\nu,\rho \in{}^{}_M \mathcal{X}^0_M $ and $f\in\text{Hom}(\kappa\omega, \nu)$,
As a result of (B) and (C), 
we have the following braiding fusion relations\label{bfe}:
\begin{align}
        \rho(f)\mathcal{E}_r^{\pm}(\kappa,\rho)\kappa(\mathcal{E}_r^{\pm}(\omega,\rho)) &= \mathcal{E}_r^\pm(\nu,\rho)f. \\
        f\kappa(\mathcal{E}_r^\pm(\rho,\omega))\mathcal{E}_r^\pm(\rho,\kappa) &= \mathcal{E}_r^\pm(\rho,\nu)\rho(f).
    \end{align}

\subsection{$Q$-systems and classification}
Let $M$ be an infinite factor. A $Q$-system \cite{longo} is a triple of an endomorphism of $M$ and isometries 
$v\in\operatorname{Hom}(\id,\gamma), w\in\operatorname{Hom}(\gamma,\gamma^2)$ satisfying the following identities:
\begin{align*}
    v^*w=\gamma(v^*)w \in\R_{>0},
    \gamma(w)w=w^2.
\end{align*}
If $N\subset M$ is a finite index subfactor, the associated canonical endomorphism gives rise to a $Q$-system.
Conversely, any $Q$-system determines a subfactor $N\subset M$ such that $\gamma$ is the canonical endomorphism:
$N$ is given by $N=\{x\in M: wx=\gamma(x)w\}$.

\subsection{Graded Local Conformal Nets}
The concept of graded local conformal nets is a generalization of the notion of local conformal nets and has been explicitly introduced and studied in \cite{ckl}.\\
Let $S^1=\{z\in\C: |z|=1\}$ be the unit circle. Let 
$\operatorname{Diff}(S^1)$ be the infinite dimensional real Lie group of orientation-preserving smooth diffeomorphisms of $S^1$ and denote by $\operatorname{Diff}(S^1)^{(n)}, n\in\mathbb{N}\cup\{\infty\}$ the corresponding $n$-cover. In particular, $\operatorname{Diff}(S^1)^{(\infty)}$ is the universal covering group of $\operatorname{Diff}(S^1)$. By identifying $\operatorname{PSL}(2,\R)$ with the group of \mob \ transformations on $S^1$ we can consider it as a three dimensional Lie subgroup of $\operatorname{Diff}(S^1)$. We denote by $\operatorname{PSL}(2,\R)^{(n)}\subset\operatorname{Diff}(S^1)^{(n)}, n\in\mathbb{N}\cup\{\infty\}$ the corresponding $n$-cover so that 
$\operatorname{PSL}(2,\R)^{(\infty)}$ is the universal covering group of $\operatorname{PSL}(2,\R)$. We denote by $\dot{g}\in\operatorname{Diff}(S^1)$ the image of $g\in\operatorname{Diff}(S^1)^{(\infty)}$ under the covering map.\\
Let $\mathcal{I}$ denote the set of all connected non-dense open intervals in $S^1$. For $I\in\mathcal{I}$, $I'$ denotes the interior of $S^1\backslash I$. Given $I\in\mathcal{I}$, the subgroup $\operatorname{Diff}(S^1)_I$of diffeomorphisms localized in $I$ is defined as the stabilizer of $I$ in $\operatorname{Diff}(S^1)$, namely the subgroup of $\operatorname{Diff}(S^1)$
whose elements are the diffeomorphisms acting trivially on $I'$.
\begin{definition}
    A graded local conformal net $\A$ on $S^1$ is a map $I\mapsto\A(I)$ from the set of intervals(connected, non dense) $\mathcal{I}$ to the set of von Neumann algebras acting on a common infinite dimensional separable Hilbert space $\hilb$ which satisfies the following properties:
    \begin{enumerate}
        \item Isotony. $\A(I)\subset\A(J)$ if $I, J\in\mathcal{I}$ and $I\subset J$.
        \item \mob \ covariance. There is a strongly continuous unitary representation $U$ of $\operatorname{PSL}(2,\R)^{(\infty)}$ such that 
        \[
        U_g\A(I)U^*_g = \A(\Dot{g}I), g\in \operatorname{PSL}(2,\R)^{(\infty)}, I \in \mathcal{I}.
        \]
        \item Positive energy. The conformal Hamiltonian $L_0$ is positive.
        \item Existence and uniqueness of the vacuum. There exists a $U$-invariant vector $\Omega\in\hilb$ which is unique up to a phase and cyclic for $\bigvee_{I\in\mathcal{I}}\A(I)$, the von Neumann algebra generated by the algebras $\A(I), I \in \mathcal{I}$.
        \item Graded locality. There exists a self adjoint unitary(the grading unitary) $\Gamma$ on $\hilb$ satisfying $\Gamma\A(I)\Gamma=\A(I)$ for all $I\in\mathcal{I}$ and $\Gamma\Omega=\Omega$ and such that $\A(I')\subset Z\A(I)'Z^*, \ I\in\mathcal{I}$, where $Z\coloneqq\frac{1-i\Gamma}{1-i}$.
        \item Diffeomorphism covariance. There is a strongly continuous projective unitary representation of $\operatorname{Diff}(S^1)^{(\infty)}$, denoted again by $U$, extending the unitary representation of $\operatorname{PSL}(2,\R)^{(\infty)}$ and such that 
        \[
         U_g\A(I)U^*_g = \A(\Dot{g}I), \ g\in \operatorname{Diff}(S^1)^{(\infty)}, \ I \in \mathcal{I},
        \]
        and such that 
        \[
         U_gxU^*_g = x, x\in\A(I'),\  g\in \operatorname{Diff}(S^1)^{(\infty)}_I,\  I \in \mathcal{I}.
        \]
    \end{enumerate}
    
\end{definition}
A local conformal net is a graded local conformal net with trivial gauge unitary $\Gamma=1$. The Bose subnet of a graded local conformal net $\A$ is defined as the fixed point subnet $\A^{\mathfrak{F}}$, where $\mathfrak{F} = \operatorname{Ad}\Gamma$ is called the Fermi grading. The restriction of $\A^{\mathfrak{F}}$ to the Bose subspace of $\hilb$ is a local conformal net with respect to the restriction to this subspace of the projective representation $U$ of $\operatorname{Diff}(S^1)^{(\infty)}$.

Some consequences \cite{ckl},\cite{fg}, \cite{gl}, \cite{cw} of the preceding definition are:
\begin{enumerate}
    \item Reeh-Schlieder Property. $\Omega$ is cyclic and separating for every $\A(I), I\in\mathcal{I}$.
    \item Bisognano-Wichmann Property. Let $I\in\mathcal{I}$ and let $\Delta_I, J_I$ be the modular operator and the modular conjugation of $(\A(I), \Omega)$. Then we have
    \[
    U(\delta_I(-2\pi t)) = \Delta_I^{it}, t\in\R.
    \]
    Moreover, the unitary representation $U:\operatorname{PSL}(2,\R)^{(\infty)}\to B(\hilb)$ extends to an (anti-)unitary representation of $\operatorname{PSL}(2,\R)\rtimes\Z/2$ determined by 
    \[
    U(r_I) = ZJ_I,
    \] acting covariantly on $\A$. Here $(\delta_I(t))_{t\in\R}$ is the one-parameter dilation subgroup of $\operatorname{PSL}(2,\R)$ with respect to $I$ and $r_I$ the reflection of the interval onto the complement $I'$.
    \item Graded Haag Duality. $\A(I')=Z\A(I)'Z^*$, for $I\in\mathcal{I}$.
    \item Outer regularity. $\A(I_0)=\bigcap_{I\in\mathcal{I}, I\supset \bar{I_0}} \A(I), I_0\in\mathcal{I}$.
    \item Additivity. If $I=\bigcup_{\alpha}I_\alpha$ with $I, I_\alpha \in \mathcal{I}$, then $\A(I)=\bigvee_\alpha\A(I_\alpha)$.
    \item Factoriality. $\A(I)$ is a type III$_1$-factor for $I\in\mathcal{I}$.
    \item Irreducibility. $\bigvee_{I\in\mathcal{I}}\A(I)=B(\hilb)$.
    \item Vacuum Spin-Statistics theorem. $e^{i2\pi L_0}=\Gamma$.
    \item Uniqueness of Covariance. For fixed $\Omega$, the strongly continuous projective representation $U$ of $\operatorname{Diff}(S^1)^{(\infty)}$ making the net covariant is unique. 
\end{enumerate}

\section{Representations of Graded Local Conformal Nets}

In this section, we assume that $\A$ is a graded local conformal net, although certain notions and results are valid in the \mob \ covariance case. 
\begin{definition}
    A Doplicher-Haag-Roberts (DHR) representation $\lambda$ of $\A$ is a map 
    $I \to \lambda_I$ that associates to an interval $I\in\mathcal{I}$ a normal representation $\lambda_I$ of $\A(I)$ on a fixed Hilbert space $\hilb_\lambda$ such that 
    \[
    \lambda_{\Tilde{I}}|_{\A(I)} = \lambda_I, \ I\subset\Tilde{I},
    \]
    for $\Tilde{I}\in\mathcal{I}$.
\end{definition}
We say that a representation $\lambda$ on $\hilb_\lambda$ is diffeomorphism covariant if there exists a projective unitary representation $U_\lambda$ of $\operatorname{Diff}(S^1)^{\infty}$ on $\hilb_\lambda$ such that
\[
\lambda_{\Dot{g}I}(U_gxU^*_g) = U_\lambda(g)\lambda_I(x)U^*_\lambda(g), \ x\in\A(I), \ g\in\operatorname{Diff}(S^1)^{\infty}.
\]
Here $\Dot{g}$ denotes the image of $g$ in $\operatorname{Diff}(S^1)$. A \mob \ covariant representation is analogously defined.\\
\begin{definition}
    A DHR representation $\lambda$ on $\hilb_\lambda$ is graded if there exists a self adjoint unitary operator $\Gamma_\lambda$ on $\hilb_\lambda$ such that 
\[
\lambda_I(\Gamma x\Gamma) = \Gamma_\lambda\lambda_I(x)\Gamma_\lambda, \ x\in\A(I),
\]
for $I\in\mathcal{I}$.
\end{definition}
\begin{proposition}[\cite{ckl}, Proposition 12]
    Let $\lambda$ be an irreducible DHR representation of $\A$. Then $\lambda$ is graded and diffeomorphism covariant.
\end{proposition}
A localized endomorphism $\lambda$ of $\A$ is an endomorphism of the universal $C^*$-algebra 
$C^*(\A)$ (\cite{be1}, Section $2.2$) such that $\lambda|_{\A(I')}$ is identity for some $I\in\mathcal{I}$. Then we say $\lambda$ is localized in the interval $I$.

\begin{definition}
    A localized endomorphism of $\lambda$ of $\A$ is called graded if it commutes with the Fermi grading $\mathfrak{F}$, namely,  $\Gamma\lambda(x)\Gamma = \lambda(\Gamma x \Gamma)$, $x\in C^*(\A)$.
\end{definition}

The following proposition generalizes to the graded local conformal net case the well known DHR argument for correspondence between DHR representations and localized endomorphisms. 

\begin{proposition}
    Let $\pi$ be a graded DHR representation of the graded local conformal net $\A$, and suppose
the Hilbert spaces $\hilb$ and $\hilb_\pi$ to be separable. Given an interval $I\in\mathcal{I}$, there exists a graded endomorphism $\rho$ of $\A$ localized in $I$ unitarily equivalent to $\pi$, $\pi=\operatorname{Ad}U\circ\rho$. 
\end{proposition}
\begin{proof}
By Proposition $14$ in \cite{ckl}, there exists a unitary operator $U: \hilb \to \hilb_{\pi}$ and a localized endomorphism $\rho$ of $\A$ such that $\pi=\operatorname{Ad}U\circ\rho$ and $U^*\Gamma_{\pi}=\Gamma U^*$. Since $\pi$ is graded, we have
\[
U\Gamma\rho(x)\Gamma U^*=\Gamma_{\pi}U\rho(x)U^*\Gamma_{\pi}=\Gamma_{\pi}\pi(x)\Gamma_{\pi} = \pi(\Gamma x\Gamma) = U\rho(\Gamma x\Gamma) U^*.
\]
Therefore, $\rho$ is a  graded localized endomorphism of $\A$.
\end{proof}

Fix an interval $I\in\mathcal{I}$. Denote by $\operatorname{Rep}^{\text{Gr}}_I(\A)$ the category of graded endomorphisms of $\A$ localized in $I$, whose object set $\Delta_I^{\text{Gr}}(\A)$ consists of graded endomorphisms of $\A$ localized in $I$ and morphisms between objects are intertwiners, i.e.  $\operatorname{Hom}_{\A(I)}(\lambda,\mu)=\{x\in Z\A(I)Z^*: x\lambda(a)=\mu(a)x, a\in\A(I)\}$. Denote by 
    $\langle\lambda,\mu\rangle_{\A(I)}$ the dimension of $\operatorname{Hom}_{\A(I)}(\lambda,\mu)$.
\begin{definition}
    A conformal net $\{M(I)\}_{I\in\mathcal{I}}$ is completely rational \cite{klm} if its representation category $\operatorname{Rep}(M)$ is a modular tensor category. 
\end{definition}

\begin{theorem}
    Let $\A$ be a graded local conformal net whose bose subnet 
    $\A_b$ is completely rational. Then $\operatorname{Rep}_I^{\text{Gr}}(\A)$ is a braided tensor category. 
\end{theorem}
\begin{proof}
    Consider the subfactor $\A_b(I) \subset \A(I)$, where $\A_b$ is the Bose subnet of $\A$. Any element in $\Delta_I^{\text{Gr}}(\A)$ is ambichiral in the sense of \cite{bek2}, section $2$. 
    Indeed, for any $\lambda\in\Delta_I^{\text{Gr}}(\A)$, its restriction $\lambda_b \coloneqq \lambda|_{\A_b}$ gives a DHR representation of the Bose subnet. Since $\lambda$ is graded, it commutes with the Fermi grading, i.e., $\lambda = \text{Ad}\Gamma\circ\lambda\circ\text{Ad}\Gamma $. By proposition 26 in \cite{ckl}, we have $\lambda = \alpha^{+}_{\lambda_b}=\alpha^{-}_{\lambda_b}$. 
    The relative braiding operators are defined by taking intertwining isometries $s,t$ to be the identity operator $\mathcal{E}_r^{\pm}(\lambda,\mu) = \epsilon^{\pm}(\lambda_b,\mu_b)$, where $\lambda,\mu\in\Delta_I^{\text{Gr}}(\A)$ and $ \epsilon^{\pm}(\lambda_b,\mu_b)$ are unitary braiding operators giving the braiding structure of the DHR representation category of $\A_b$. 
    \end{proof}
 In the rest of this paper, we write $\mathcal{E}_r^{\pm}(\lambda,\mu)$ as $\epsilon^{\pm}(\lambda,\mu)$, if this does not cause any ambiguity.

\begin{remark}
    The braided tensor category structure on the representation theory of vertex operator (super)algebras is known and studied in section $2$ of \cite{ckm}. In our case, the structure morphisms of the $Q$-system are even morphisms satisfying identities in definition $2.25$ of \cite{ckm}. However, in this paper, we mainly focus on the graded locality and its influence on counting multiplicities of the dual canonical endomorphism hence we do not study the supercateogry structure in detail. 
\end{remark}

\begin{remark}
    In \cite{cgh}, to find a graded local extension $\A\subset\B $ for a given graded local conformal net $\A$, they first fix a $Q$-system $\Theta_b$ in $\operatorname{Rep}(\A_b)$. $\Theta_b$ determines the inclusion $\A_b\subset\B_b$.
    Then they search for another $Q$-system $\Theta_f$ which is a $\Theta_b$-module in $\operatorname{Rep}(\A_b)$ with twist $-1$. $\Theta_f$ corresponds to the extension $\B_b\subset\B$. The expected extension $\A\subset\B$ is given by the $Q$-system 
    $\Theta_b\oplus\Theta_f$, which is the graded conformal net analogous to the superalgebra object in the sense of definition $2.25$ of \cite{ckm}. In addition, as in our case, 
    we directly carry out $\alpha$-induction on $\A_b\subset\B$. Therefore, our approach and \cite{cgh} are two different points of view of the same $Q$-system. 
\end{remark}

\section{Simple Current Extensions of Graded Local Conformal Nets}
In this section, we study simple current extensions of graded local conformal nets. Let $(\A, \hilb_0, \Omega_0, \Gamma_0)$ be a graded local net. Assume $\sigma$ is an order $n$ automorphism in $\Delta_{I_0}^{\text{Gr}}(\A)$. Extend $\sigma$ such that $\sigma(\Gamma_0)=\Gamma_0$. Define a representation $\pi: \A \to B(\hilb)$ by:
\begin{equation}
    \pi(a)\xi \coloneqq \oplus_{p=0}^{n-1}\sigma^p(a)\xi_p \ \text{and} \ \pi(\Gamma_0)\xi\coloneqq\oplus_{p=0}^{n-1}\Gamma_0\xi_p,
\end{equation}
where $\hilb = \hilb_0^{\oplus n}$ and $\xi = \oplus_{p=0}^{n-1} \xi_p$. 
Denote $\B(I)\coloneqq \pi(\A(I))$, $\Omega_p\coloneqq \delta_{p,0}\Omega_0$, $\Gamma \coloneqq \oplus_{p=0}^{n-1}\Gamma_0$, $Z \coloneqq \frac{1-i\Gamma}{1-i}$. Then $\Gamma\pi(x)\Gamma = \pi(\Gamma_0x\Gamma_0)$. For any interval $I$, we can take a unitary operator $u_I$ such that $[u_I,\Gamma_0] = 0$ and 
$\text{Ad}(u_I)\circ\sigma$ is a graded endomorphism localized in $I$.
Define $f_{u_I}\in B(\hilb)$, by 
\begin{equation}\label{defequa}
    (f_{u_I}\xi)_p \coloneqq \sigma^{p-1}(u^*_I)\xi_{p-1}, \ p\in \Z_n.
\end{equation}
As a result,
\[
(f_{u_I}^*\xi)_p = \sigma^p(u_I)\xi_{p+1}.
\]
Then, we have $f_{u_I}^*\pi(a)f_{u_I} = \pi\circ\sigma_I(a)$ for any $a\in\A$. Note that 
$f_{u_I}\in\B(I')'$ since $\B(I)$ may not be Haag dual. Now we define the extended net $\mathcal{C}$ on $\hilb$ as
\[
\mathcal{C}(I) \coloneqq \langle\B(I), Z^*f_{u_I}Z\rangle'',
\]
the von Neumann algebra generated by $\B(I)$ and $Z^*f_{u_I}Z$.
Note that $f_{u_I} = f_{\id}\pi(u_I^*)$, where $f_{\id}$ is defined by replacing $u_I$ with the identity operator in (\ref{defequa}). The definition of the extended net is independent of the choice of $u_I$. Indeed, if $\text{Ad}(\hat{u_I})\circ\sigma \in \Delta_I^{\text{Gr}}(\A)$, for another $\hat{u_I}$, we have 
$u_I\hat{u_I}^*\in\A(I')' = Z_0\A(I)Z_0^*$.
Then $Z^*f_{u_I}Z$ and $Z^*f_{\hat{u_I}}Z$ only differ by an element of $\B(I)$ because $\pi(Z_0aZ_0^*) = Z\pi(A)Z^*$. In addition, 
we have:
\begin{equation}
    \Gamma f_{u_I}\Gamma = \Gamma f_{\id}\pi(u_I^*)\Gamma = f_{\id}\pi(\Gamma_0u_I^*\Gamma_0)= f_{u_I}.
\end{equation}
where we use the fact that $[\Gamma,f_{\id}]=[\Gamma_0,u_I]=0$. Therefore $[Z,f_{u_I}]=0$ and 
we can rewrite $\mathcal{C}(I) = \langle\B(I), f_{u_I}\rangle''$.

\begin{lemma}
\label{lemma:1}
    In the above setting, we have, 
    $f_{\id}\pi(u_I^*) = \pi(\sigma^{-1}(u_I^*))f_{\id}$.
\end{lemma}
\begin{proof}
    Take $\xi\in\hilb$. Then $(f_{\id}\pi(u_I^*))\xi)_p = (\pi(u_I^*)\xi)_{p-1} = \sigma^{p-1}(u_I^*)\xi_{p-1}$. On the other hand, 
    $(\pi(\sigma^{-1}(u_I^*))f_{\id}\xi)_p = \sigma^{p-1}(u_I^*)(f_{\id}\xi)_p=\sigma^{p-1}(u_I^*)\xi_{p-1}$.
\end{proof}
\begin{lemma}[\cite{be2} lemma $3.6$]
     In the above setting, we have, 
    $f_{u_{I_2}}f_{u_{I_1}} = \epsilon^{\pm}(\sigma,\sigma)f_{u_{I_1}}f_{u_{I_2}}$, if 
    $I_1\cap I_2 = \emptyset$.
\end{lemma}
\begin{proof}
First notice that $f_{u_I} = f_{\id}\pi(u_I^*)$ for $I\in\mathcal{I}$.
    \begin{align*}
        f_{u_{I_2}}f_{u_{I_1}} &= f_{\id}\pi(u_{I_2}^*)f_{\id}\pi(u_{I_1}^*)\\
        &= \pi(\sigma^{-1}(u^*_{I_2})\sigma^{-2}(u^*_{I_1}))f^2_{\id} \\
        &= \pi(\sigma^{-1}(u^*_{I_2})\sigma^{-2}(u^*_{I_1})\sigma^{-2}(u_{I_2})\sigma^{-1}(u_{I_1}))f_{\id}\pi(u_{I_1})^*f_{\id}\pi(u_{I_2})^*\\
        &= \sigma^{-2}\circ\pi(\sigma(u^*_{I_2})u^*_{I_1}u_{I_2}\sigma(u_{I_1}))f_{u_{I_1}f_{u_{I_2}}}\\
        &= \sigma^{-2}(\epsilon^{\pm}(\sigma,\sigma))f_{u_{I_1}f_{u_{I_2}}}
    \end{align*}
    where we repeatedly use lemma \ref{lemma:1}.
\end{proof}

\begin{proposition}
    The extension net $\{\mathcal{C}(I)\}_{I\in\mathcal{I}}$ is graded local if and only if $\epsilon(\sigma,\sigma)^{\pm}=1$.
\end{proposition}
\begin{proof}
    Since $\A(I)$ is relatively graded local in $\B(I)$, it only suffices to check
    \begin{equation}
        f_{u_{I_2}}f_{u_{I_1}} = f_{u_{I_1}}f_{u_{I_2}},
    \end{equation}
    which is guaranteed by the above lemma.
\end{proof}
To sum up the preceding argument, we have the following result:
\begin{theorem}
    The simple current extension net $\{\mathcal{C}(I)\}_{I\in\mathcal{I}}$ is a graded local conformal net if and only if $\epsilon(\sigma,\sigma)^{\pm}=1$.
\end{theorem}

\section{From Inclusion of Nets to Modular Invariants}
Let $\A \subset \B$ be an inclusion of graded local conformal nets on
$(\hilb, \Omega)$ with the same gauge unitary $\Gamma$. Let $\gamma: \B(I)\to\A(I)$ be the canonical endomorphism and $\theta=\gamma|_{\A(I)}$ be the dual canonical endomorphism. 
Then $\B(I) = \A(I)v$, where\\
$ v\in\operatorname{Hom}_{\B(I)}(\operatorname{id}, \gamma)$ and $w\in \operatorname{Hom}_{\A(I)}(\operatorname{id}, \theta)$ are isometries such that $w$ induces the conditional expectation 
$E(b) = w^*\gamma(b)w, b\in\B(I)$  and 
$w^*v = |\B : \A|^{\frac{1}{2}} = w^*\gamma(v)$.
Here, $|\B : \A|$ is the index of the subfactor $\A(I)\subset\B(I)$. The triple $(\theta, v, w)$ as above with some compatible conditions defines a $Q$-system.
In the following, we set $\epsilon(\lambda,\mu) = \epsilon^+(\lambda,\mu)$. Arguments for $\epsilon^+(\lambda,\mu)$ also work for $\epsilon^-(\lambda,\mu)$.
\begin{lemma}[\cite{be1} lemma $3.1$]\label{lembfe}
    $\theta\in\Delta_I^{\text{Gr}}(\A)$
    and for $\lambda\in\Delta_I^{\text{Gr}}(\A)$, we have 
    \[
    \operatorname{Ad}(\epsilon(\lambda,\theta))\circ\lambda\circ\gamma(v) = \theta(\epsilon(\lambda,\theta)^*)\gamma(v).
    \]
\end{lemma}
\begin{proof}
    $\theta$ is a direct sum of irreducible DHR representations of $\A$ and they are graded by proposition $1$.
    By proposition 3, all graded localized endomorphisms are ambichiral which implies the braiding fusion equation relations $(1)$ and $(2)$. By the intertwining property of $v$, 
    $\gamma(v)^*\in\text{Hom}_{\A(I)}(\theta\circ\theta,\theta)$. Applying this to the braiding fusion relation $(1)$, we have:
    \[
    \gamma(v)^*\theta(\epsilon(\lambda,\theta))\epsilon(\lambda,\theta) = \epsilon(\lambda,\theta)\lambda(\gamma(v)^*),
    \]
    hence,
    \[
    \gamma(v)^*\theta(\epsilon(\lambda,\theta)) = \text{Ad}(\epsilon(\lambda,\theta))\circ\lambda(\gamma(v)^*). 
    \]
\end{proof}

\begin{proposition}\label{commu}
    Take $\theta, \gamma, v$ as above. We have
    \[
    \epsilon(\theta, \theta)v^2 = \epsilon(\theta, \theta)^*v^2 = v^2.
    \]
    \[
     \epsilon(\theta, \theta)\gamma(v) =  \epsilon(\theta, \theta)^*\gamma(v) = \gamma(v).
    \]
\end{proposition}
\begin{proof}
    Take $I_+$ another interval disjoint with $I$. Take $u\in\A$ which intertwines $\gamma$ and $\gamma_+$. Particularly, $uv$ intertwines  $\theta$ and $\theta_+$. Moreover,
    $uv$ intertwines $\text{id}_{\B}$ and $\theta_+$ which implies $uv\in\B(I'_+)'$ and $[uv,v]=0$.
    $uvv=vuv=\theta(u)v^2$ with $\epsilon(\theta,\theta)=u^*\theta(u)$, we get $\epsilon(\theta,\theta)v^2=u^*\theta(u)v^2=v^2$. By the fact that $v^2=\gamma(v)v$, we have $\epsilon(\theta,\theta)\gamma(v)vv^*=\gamma(v)vv^*$, taking conditional expectation on both sides, we get $\epsilon(\theta,\theta)\gamma(v)=\gamma(v)$.
\end{proof}

\begin{lemma}\label{ext}
    Let $t\in Z\B(I)Z^*$ such that $t\lambda(a) = \mu(a)t$ for any $a\in\A(I)$, then $t\in\text{Hom}_{\B(I)}(\alpha_\lambda, \alpha_\mu)$. 
\end{lemma}
\begin{proof}
    Let $s=\gamma(t)$. Then clearly $s\in\operatorname{Hom}_{\A(I)}(\theta\circ\lambda,\theta\circ\mu)$. By proposition \ref{bfe}, we have 
    \[
    \epsilon(\theta\circ\mu,\theta)s = \theta(s)\epsilon(\theta,\theta)\theta(\epsilon(\lambda,\theta)).
    \]
    Since $\epsilon(\theta\circ\mu,\theta)=\epsilon(\theta,\theta)\theta(\epsilon(\lambda,\theta)$, we have \[
    s\theta(\epsilon(\lambda,\theta)^*) = \theta(\epsilon(\mu,\theta)^*)\epsilon(\theta,\theta)^*\theta(s)\epsilon(\theta,\theta).
    \]
    By repeatedly using proposition \ref{commu} and lemma \ref{lembfe} we have, 
    \begin{align*}
        s\operatorname{Ad}(\epsilon(\lambda,\theta))\circ\lambda\circ\gamma(v) &= s\theta(\epsilon(\lambda,\theta)^*)\gamma(v)\\
        &= \theta(\epsilon(\mu,\theta)^*)\epsilon(\theta,\theta)^*\theta(s)\epsilon(\theta,\theta)\gamma(v)\\
        &= \theta(\epsilon(\mu,\theta)^*)\epsilon(\theta,\theta)^*\theta(s)\gamma(v)\\
        &= \theta(\epsilon(\mu,\theta)^*)\epsilon(\theta,\theta)^*\gamma(v)s\\
        &= \theta(\epsilon(\mu,\theta)^*)\gamma(v)s\\
        &= \operatorname{Ad}(\epsilon(\mu,\theta))\circ\mu\circ\gamma(v)s.
    \end{align*}
    Taking $\gamma^{-1}$ on both sides of the above identity, we obtain $t\alpha_\lambda(v) = \alpha_\mu(v)t$.
\end{proof}
\begin{lemma}[\cite{be1} lemma $3.8$]\label{inj}
For $a\in\A(I)$, $av=0$ implies $a=0$. For $b\in\B(I)$, $w^*\gamma(b)=0$ implies $b=0$.    
\end{lemma}
\begin{theorem}[\cite{be1}]
    For $\lambda, \mu \in \Delta_I^{\text{Gr}}(\A)$, $ \langle\alpha_\lambda, \alpha_\mu \rangle_{\B(I)} = \langle\theta\circ\lambda, \mu \rangle_{\A(I)} $.
\end{theorem}
\begin{proof}
    Let $x\in\operatorname{Hom}_{\B(I)}(\alpha_\lambda,\alpha_\mu)$. Then, for $a\in\A(I)$,
    \[
    w^*\gamma(x)\theta\circ\lambda(a) = w^*\gamma(x\lambda(a)) = w^*\gamma(\mu(a)x)= 
    w^*\theta\circ\mu(a)x = \mu(a)w^*\gamma(x).
    \]
    Here, we use the intertwining property of $x$ and $w$. The above identity means $w^*\gamma(x)\in\operatorname{Hom}_{\A(I)}(\theta\circ\lambda,\mu)$ and the map $x\to w^*\gamma(x)$ is injective by lemma \ref{inj}.\\
    On the other hand, let $y\in\operatorname{Hom}_{\A(I)}(\theta\circ\lambda,\mu)$. Then, for $a\in\A(I)$,
    \[
    yv\lambda(a) = y \theta\circ\lambda(a)v = \mu(a)yv.
    \] Here, we use the intertwining property of $v$ and $y$. By lemma \ref{ext}, we have\\ 
    $yv\in\operatorname{Hom}_{\B(I)}(\alpha_\lambda,\alpha_\mu)$ and by lemma \ref{inj}, the map $y\to yv$ is injective. 
\end{proof}
\begin{corollary}

   In the above setting, define the modular invariant matrix $Z_{\lambda,\mu} = \langle\alpha^+_\lambda, \alpha^-_\mu\rangle_{\B(I)}$, where $\lambda,\mu\in\Delta_I^{\text{Gr}}(\A)$. Then, we have $Z_{0,\mu} = \langle\theta,\mu\rangle_{\A(I)}$ for any $\mu\in\Delta_I^{\text{Gr}}(\A)$, where 
   $0$ denotes the vacuum representation.
\end{corollary}

\section{$N=2$ Super-Virasoro Nets and Their Representations}
\begin{definition}
For any $t\in\R$, the $N=2$ super-Virasoro algebra $\operatorname{SVir}^{N=2,t}$ is the infinite dimensional Lie superalgebra generated by linearly independent even elements $L_n, J_n$ and odd elements $G_r^{\pm}$ where $n\in\Z$, $r, s\in\frac{1}{2}\mp t+\Z$, together with an even central element $c$ and with commutation relations:
    \begin{align*} 
    [L_m,L_n] &= (m-n)L_{m+n} + \frac{c}{12}(m^3-m)\delta_{m+n,0},\\
    [L_m,J_n] &= -nJ_{n+m},\\
    [J_m,J_n] &= \frac{c}{3}m\delta_{m+n,0},\\
    [L_n,G_r^{\pm}] &= (\frac{n}{2}-r)G^\pm_{n+r},\\
    [J_n,G_s^{\pm}] &= \pm\frac{1}{2}G^\pm_{n+s},\\
    [G_r^+,G_s^+] &= [G_r^-,G_s^-] =0,\\
    [G_r^+,G_s^-] &= 2L_{r+s}+(r-s)J_{r+s}+\frac{c}{3}(r^2-\frac{1}{4})\delta_{r+s,0}.
\end{align*}
\end{definition}
The Neveu-Schwarz(NS) $N=2$ super-Virasoro algebra is the super-Virasoro algebra with $t=0$, while the Ramond(R) super-Virasoro algebra is the one with $t=\frac{1}{2}$.\\
We are interested in irreducible unitary representations. Here, we mainly consider Neveu-Schwarz super-Virasoro algebra since it has the vacuum representation. In this case, the algebra of zero modes is abelian and the irreducibility implies that the lowest energy space is one dimensional and it is spanned by a single unit vector $\Omega_{c,h,q}$ such that $L_0\Omega_{c,h,q}=h\Omega_{c,h,q}$ and $J_0\Omega_{c,h,q}=q\Omega_{c,h,q}$. The real numbers $c, h, q$ completely determine the representation up to unitary equivalence.
\begin{theorem}[\cite{urep1}, \cite{urep2}, \cite{urep3}]
    For any irreducible unitary representation of the Neveu-Schwarz $N=2$ super-Virasoro algebra the corresponding values of $c, h, q$ satisfy one of the following conditions:
    \begin{enumerate}
        \item NS1 $c\geq3$ and $2h-2nq+(\frac{c}{3}-1)(n^2-\frac{1}{4})\geq0$ for all $n\in\frac{1}{2}+\Z$.
        \item  NS2 $c\geq3$ and $2h-2hq+(\frac{c}{3}-1)(n^2-\frac{1}{4})=0$,\\
        $2h-2(n+\operatorname{sgn}(n))q+(\frac{c}{3}-1)[(n+\operatorname{sgn}(n))^2-\frac{1}{4}]<0$ for some $n\in\Z+\frac{1}{2}$ and $2(\frac{c}{3}-1)h-q^2+\frac{c}{3}\geq0$.
        \item $c=\frac{3n}{n+2}, h=\frac{l(l+2)-m^2}{4(n+2)}, q=\frac{-m}{n+2}$, where 
        $n, l, m\in\Z$ satisfy $n\geq0, 0\leq l\leq n, l+m\in2\Z$ and $|m|\leq l$.
    \end{enumerate}
    Conditions NS1, NS2, NS3, are also sufficient, namely if values $c, h, q$ satisfy one of them then there exists a corresponding  irreducible unitary representation. In particular, all the values in the discrete series of representations (NS3) with $c=\frac{3n}{n+2}$ are realized by the coset construction for the inclusion $U(1)_{2n+4}\subset SU(2)_{n}\otimes \operatorname{CAR}^{\otimes2}$ for every nonnegative integer $n$. Here $\operatorname{CAR}^{\otimes2}$ denotes the theory generated by two real chiral free Fermi fields.
\end{theorem}

In fact, associating to the vacuum representation with central charge $c$ of the Neveu-Schwarz $N=2$ super-Virasoro algebra for every allowed $c$, we can define a graded local conformal net, which leads to the definition of $N=2$ superconformal net. This is defined in \cite{chklx} definition $3.5$ by smearing certain fields operators.
\begin{definition}
    The $N=2$ super-Virasoro net with central charge $c$ is defined as follows:
    \[
    \A_c(I) = \{e^{iJ(f)},e^{iL(f)},e^{iG^k(f)}: f\in C_c^{\infty}(I,\R), k=1,2 \}'', I\in\mathcal{I}.
    \]
    Here, 
    \[
G^1_r \coloneqq \frac{G^+_r+G^-_r}{\sqrt{2}}, \ G^2_r \coloneqq -i \frac{G^+_r-G^-_r}{\sqrt{2}}, \ r\in \frac{1}{2}+\Z. 
\]
\end{definition}

\begin{definition}
    An $N=2$ superconformal net is a graded local conformal net $\A$ with central charge $c$ containing the $N=2$ super-Virasoro net $\A_c$ as a covariant subnet and such that the corresponding representations of $\operatorname{Diff}(S^1)^{(\infty)}$ agree. 
\end{definition}
For any allowed central charge $c = \frac{3n}{n+2}$ in the discrete series,\\
$\A_c(I) = \A_{U(1)_{2n+4}}(S^1)^{'}\bigcap(\A_{SU(2)_n}(I)\otimes\A_{\text{CAR}^{\otimes 2}}(I))$, by theorem $5.10$ in \cite{chklx}.
DHR representations are labeled by $(l,m)$ satisfying $l=0,1,2, ..., n$, $m=0,1,2,...,2n+3\in\Z/(2n+4)\Z$, with $l-m\in 2\Z$ with the identification $(l,m)=(n-l,m+n+2)$.\\
Fusion rules:
\[
(l_1,m_1)(l_2,m_2) = \bigoplus_{|l_1-l_2|\leq l \leq \text{min}\{l_1+l_2,2n-l_1-l_2\},
l+l_1+l_2\in 2\Z}(l,m_1+m_2).
\]
The restriction of $(l,m)$ to the Bose subnet is given by $(l,m,0)\oplus(l,m,2)$,  \\
$(l,m) = \alpha_{(l,m,0)}$. Statistics phases and dimensions:
\[\omega_{(l,m)} = \exp{(\frac{l(l+2)-m^2}{4(n+2)}2\pi i)}, 
\ d_{(l,m)} = \frac{\sin\frac{(l+1)\pi}{n+2}}{\sin\frac{\pi}{n+2}}.
\]

\section{Classification of $N=2$ Superconformal Nets}

According to Gannon's list \cite{gannonlist}, extensions of $N=2$ super-Virasoro nets split into two parts: simple current extensions of some subgroups of the maximal cyclic group and exceptional ones. Firstly, we discuss the set consisting of DHR representations with dimension $1$ and statistics $1$. Note that this set may not be a group in general. 

\begin{enumerate}
    \item Assume $n$ is odd. Then $G=\Z_{2(k+2)}$, with the generator $\sigma = (n,1)$, and relations $(n,1)^{2N}=(0,2N),\ (n,1)^{2N+1} = (n,2N+1)$, $N\in\mathbb{N}$.
    \item Assume $n$ is even. With relations $(n,0)^2 = (0,0),
    (n,0)(0,2N) = (n,2N)$, we have
    $G \cong \Z_{n+2}\times\Z_2$, through the group homomorphism: \\$(n,2N)\mapsto(0,2N)\times(n,0), (0,2N)\mapsto(0,2N)\times(0,0)$. The generator for $\Z_{n+2}$ is $(0,2)$ and the generator for $\Z_2$ is $(n,0)$.
\end{enumerate}
Therefore, if we want to take a simple current extension for some subgroup $H$, this $H$ must be a subgroup of one of the above groups. Furthermore, we have to make sure the generating representations have statistics phase $1$.
\begin{enumerate}
    \item[(A)] When $k$ is odd, take the smallest positive integer $p$ such that $\omega_{(k,1)^p} = 1$. The maximal cyclic group is $G = \langle(k,1)^p\rangle$, which is the cyclic group generated by $(k,1)^p$.
    \item[(B)] When $k\equiv 0\pmod{4}$, $\omega_{(k,0)} =1$. Let $p$ be the smallest positive integer such that $\omega_{(0,2)^p}=1$. The maximal cyclic group is $G= \langle(0,2)^p\rangle\times\langle(k,0)\rangle$.
    \item[(C)] When $k\equiv2\pmod{4}$, $\omega_{(k,0)}=-1$, which implies that $\omega_{(0,\frac{M}{2})}\in\{-1,i,-i\}$.
    Let $M$ be the smallest positive even integer such that $\omega_{(0,M)}=1$. We have $\omega_{(0,\frac{M}{2})} = -1$, $\omega_{(k,\frac{M}{2})}=1$. Indeed, by assumption, $\frac{M^2}{8(k+2)}$ and $\frac{k}{2}$ are odd then
    \begin{equation}
        \omega_{(k,\frac{M}{2})} = \exp(\frac{k}{2}\pi i-\frac{M^2}{8(k+2)}\pi i) = 1.
    \end{equation}
    
    Therefore, the maximal cyclic group $G$ is 
    \[
    \{(0,0),(k,\frac{M}{2}), (0,M), \cdots\} = \langle(k,\frac{M}{2})\rangle \cong H \times \Z_2.
    \]
    If $\omega_{(0,\frac{M}{2})} = \pm i$, we have $\omega_{(k,\frac{M}{2})} = \pm i$. Consequently, the maximal cyclic group is 
    $H = \langle (0,M)\rangle$. 
\end{enumerate}
Finally, for exceptional cases, we have: 
\begin{enumerate}
    \item[(a)] $(E_6)$ for $k=10$, $\theta = (0,0)\oplus(6,0)$. This is a dual canonical endomorphism because it arises from a conformal embedding $SU(2)_{10}\subset SO(5)_1$ as in \cite{kl}. The $Q$-system structure on $\theta$ is unique up to unitary equivalence.
    \item[(b)] $(E_6)$ for $k=10$, $\theta=(0,0)\oplus(0,12)$, This is a dual canonical endomorphism because it arises from a conformal embedding $U(1)_{12}\subset U(1)_3$. The DHR representation $(0,12)$ has dimension $1$ and statistics phase $1$, so it is realized as a crossed product by $\Z/2\Z$ and hence the $Q$-system structure is unique up to unitary equivalence. 
    \item[(c)] $(E_6)$ for $k=10$, $\theta = (0,0)\oplus(6,0)\oplus(0,12)\oplus(6,12)$. This is the combination of the above two extensions. Namely, we first consider the extension in (a) and make a crossed product extension by $\Z/2\Z$ as in (b). For the above two reasons, the $Q$-system structure of $\theta$ is unique up to unitary equivalence.  
    \item[(d)] $(E_8)$ For $k=28$, $\theta = (0,0)\oplus(10,0)\oplus(18,0)\oplus(28,0)$. This is a dual canonical endomorphism because it arises from a conformal embedding  $SU(2)_{28}\subset (G_2)_1$. As in (a), the $Q$-system structure of $\theta$ is unique up to unitary equivalence. 
\end{enumerate}

\begin{theorem}
    
    The complete list of $N=2$ superconformal nets with $c<3$ in the discrete series is listed as follows:
    \begin{enumerate}
        \item A simple current extension arising from a subgroup of the maximal cyclic group appearing in the above (A), (B) and (C). 
        \item The exceptional cases related to $E_6$ and $E_8$ as in the above (a), (b), (c) and (d).
    \end{enumerate}
\end{theorem}

\section*{Acknowledgment}
The author would thank Yasuyuki Kawahigashi for bringing about this question and for his constant support and advice. The author would thank Sebastiano Carpi and Roberto Longo for reading the draft and for their suggestions. The author also thank the anonymous referee for the advice on the manuscript.

\end{document}